\documentclass{amsart}

\usepackage{microtype}

\usepackage{graphicx}
\usepackage{tikz}
\usetikzlibrary{arrows,decorations.pathmorphing,backgrounds,positioning,fit,petri,chains, matrix,scopes}
\usepackage[T1]{fontenc}
\usepackage[latin1]{inputenc}
\usepackage{tabularx}
\usepackage{caption}
\usepackage{subcaption}
\usepackage{nameref}
\usepackage{amssymb}
\usepackage{xcolor}

\usepackage{a4wide}

\newtheorem{theorem}{Theorem}
\newtheorem{lemma}{Lemma}

\newtheorem{corollary}{Corollary}
\newtheorem{proposition}{Proposition}

\theoremstyle{definition}
\newtheorem{definition}{Definition}

\theoremstyle{remark}
\newtheorem{assumption}{Assumption}
\newtheorem{remark}{Remark}

\newcommand{\cZ}{\mathcal{Z}}
\newcommand{\cA}{\mathcal{A}}
\newcommand{\cU}{\mathcal{U}}
\newcommand{\cL}{\mathcal{L}}
\newcommand{\cS}{\mathcal{S}}
\newcommand{\cK}{\mathcal{K}}

\newcommand{\cB}{\mathcal{B}}
\newcommand{\seq}{\textsc{Seq}}
\newcommand{\nti}{n\to\infty}
\newcommand{\eps}{\varepsilon}
\renewcommand{\(}{\left(}
\renewcommand{\)}{\right)}

\newcommand{\zTheta}[1] {\mathrm{\Theta}\left(#1\right)}
\newcommand{\BigO}[1] {O\left(#1\right)}
\newcommand{\smallO}[1] {o\left(#1\right)}

\title{Enumerating Lambda Terms by Weighted Length of Their De Bruijn Representation}
\thanks{A preliminary version of this work appeared in the proceedings of STACS 2016.}
\thanks{The work was supported by FWF grant SFB F50-03.}

\author{Olivier Bodini \and Bernhard Gittenberger \and Zbigniew Go{\l}\k{e}biewski}

\address{
Laboratoire d'Informatique de Paris-Nord, Université de Paris-Nord, 99, avenue
Jean-Baptiste Clément, 93430 Villetaneuse, France.}
\email{olivier.bodini@lipn.univ-paris13.fr}
\address{
Institute for Discrete Mathematics and Geometry, Technische Universit\"at Wien, 
Wiedner Hauptstra\ss e 8-10/104, A-1040 Wien, Austria.} 
\email{gittenberger@dmg.tuwien.ac.at}
\address{
Department of Computer Science, Wroc\l aw University of Science and Technology, ul. Wybrzeze
Wyspianskiego 27, 50-370, Wroc\l aw, Poland.}
\email{zbigniew.golebiewski@pwr.edu.pl}

\keywords{lambda term; asymptotic enumeration; generating function; infinitely nested radical;
Boltzmann sampling}

\begin{document}

\begin{abstract}
John Tromp introduced the so-called 'binary lambda calculus' as a way to encode lambda terms in
terms of 0-1-strings using the de Bruijn representation along with a weighting scheme. 
Later, Grygiel and Lescanne conjectured that the number of binary lambda terms with $m$ free
indices and of size $n$ (encoded as binary words of length $n$ and according to Tromp's weights) is $\smallO{n^{-3/2} \tau^{-n}}$ for $\tau \approx 1.963448\ldots$. 
We generalize the proposed notion of size and show that for several classes of lambda
terms, including binary lambda terms with $m$ free indices, the number of terms of size $n$ is 
$\zTheta{n^{-3/2} \rho^{-n}}$ with some class dependent constant $\rho$, which in particular
disproves the above mentioned conjecture. 

The methodology used is setting up the generating functions for the classes of lambda terms. These
are infinitely nested radicals which are investigated then by a singularity analysis. 

We show further how some properties of random lambda terms can be analyzed and present a way to
sample lambda terms uniformly at random in a very efficient way. This allows to generate terms of
size more than one million within a reasonable time, which is significantly better than the
samplers presented in the literature so far.
\end{abstract}

\maketitle

\section{Introduction}

The objects of our interest are lambda terms, which are the basic objects of lambda calculus. For
a thorough introduction to lambda terms and lambda calculus we refer to \cite{Ba84}. This paper
will not deal with lambda calculus and no understanding of lambda calculus is needed to follow our
proofs. We will instead be interested in the enumeration of lambda terms, the study of some
properties of random lambda terms and the efficient generation of terms of a given size uniformly
at random. 

A lambda term is a formal expression which is described by the grammar 
$M \, ::= \, x \, | \, \lambda x . M \, | \, (M \, M)$ 
where $x$ is a variable, the operation $(M \, M)$ is called application, and using the quantifier $\lambda$ is called abstraction.
In a term of the form $\lambda x . M$ each occurrence of $x$ in $M$ is called a bound variable.
We say that a variable $x$ is free in a term $M$ if it is not in the scope of any abstraction.
A term with no free variables is called closed, otherwise open. 
Two terms are considered equivalent if they are identical up to
renaming of the variables, \emph{i.e.}, more formally speaking, they can be
transformed into each other by $\alpha$-conversion. We shall always mean 'equivalence
class w.r.t. $\alpha$-conversion' whenever we write 'lambda term'.   

In this paper we are interested in counting lambda terms whose size corresponds to their De
Bruijn representation (\emph{i.e.}, what was called 'nameless expressions' in  
\cite{deBruijn:IndMath:1972}).

\begin{definition}
A De Bruijn representation is a word described by the following specification:
\[
M \, ::= \, n \, | \, \lambda M \, | \, M \, M 
\]
where $n$ is a positive integer, called a De Bruijn index. 
Each occurrence of a De Bruijn index is called a variable and each
$\lambda$ an abstraction. A variable $n$ of a De Bruijn representation $w$ 
is bound if the prefix of $w$
which has this variable as its last symbol contains at least $n$ times the symbol $\lambda$,
otherwise it is free. The abstraction which binds a variable $n$ is the $n$th $\lambda$ before the
variable when parsing the De Bruijn representation from that variable $n$ backwards to the first symbol. 
\end{definition}

For the purpose of the analysis we will use the notation consistent with the one used in~\cite{BGLZ16}. 
This means that the variable $n$ will be represented as a sequence of $n$ symbols, namely as a
string of $n-1$ so-called 'successors' $S$ and a so-called 'zero' $0$ at the end. Obviously, there
is a one to one correspondence between equivalence classes of lambda terms (as described in the
first paragraph) and De Bruijn representations. 
For instance, the De Bruijn representation of the lambda term 
$\lambda x . \lambda y . x y$ (which is e.g. equivalent to $\lambda a . \lambda b . a b$ or
$\lambda y . \lambda x . y x$) is $\lambda \lambda 2 1$; using the notation with successors this
becomes $\lambda \lambda ((S 0) 0)$. 

Since we are interested in counting lambda terms of given size we have to specify what we mean by
'size': We use a general 
notion of size which covers several previously studied models from the literature. The 
building blocks of lambda terms,  zeros, successors, abstractions and applications, contribute 
$a, b, c$ and $d$, respectively, to the total size of a lambda term. 
Formally, if $M$ and $N$ are lambda terms, then 
\[
|0|=a, \qquad | S n | = |n| + b, \qquad |\lambda M| = |M| + c, \qquad |M N| = |M| + |N| + d.
\]

Thus for the example given above we have 
$| \lambda \lambda ((S 0) 0) | = 2 a + b + 2 c + d.$ 
Assigning sizes to the symbols like above covers several previously introduced notions of size:
\begin{itemize}
	\item so called 'natural counting' (introduced in~\cite{BGLZ16}) where $a = b = c = d = 1$,
	\item so called 'less natural counting' (introduced in~\cite{BGLZ16}) where $a = 0, b = c = 1, d = 2$.
 	\item binary lambda calculus (introduced in~\cite{tromp:DSP:2006:628}) where $b = 1, a = c = d = 2$,
\end{itemize}

\begin{assumption}\label{i:ass:1}
Throughout the paper we will impose the following assumptions on the constants $a, b, c, d$:

\noindent
\begin{tabular}{llll}
{\bf 1.} $a,b,c,d$ are nonnegative integers, & {\bf 2.} $a+d \geq 1$, &
{\bf 3.} $b, c \geq 1$, & {\bf 4.} $\gcd(b,c,a+d) = 1$.
\end{tabular}
\end{assumption}

If the zeros and the applications both had size $0$ (\emph{i.e.} $a+d=0$), then we would have
infinitely many terms of the given size, because one could insert arbitrarily many applications
and zeros into a term without increasing its size. If the successors or the abstractions had size
$0$ (\emph{i.e.} $b$ or $c$ equals to $0$), then we would again have infinitely many terms of
given size, because one could insert arbitrarily long strings of successors or abstractions into
a term without increasing its size. The last assumption is more technical in its nature. It
ensures that the generating function associated with the sequence of the number of lambda terms
will have exactly one singularity on the circle of convergence, which is on the positive real
line. The case of several singularities is not only technically more complicated, but it is for
instance not even \emph{a priori} clear which singularities are important and which are
negligible. So we cannot expect that it differs from the single singularity case only by a
multiplicative constant.

We mention that in \cite{GL13} lambda terms with size function corresponding to $a=b=0$ and $c=d=1$ were considered, but another restriction was imposed on the terms.  

\subsubsection*{Historical remarks.} Of course, the enumeration of combinatorial structures or
the study of random structures is of interest in its own right. However, there is rising interest
in enumeration problems related to structures coming from logic. One of the first works in such
a direction were the investigation of random Boolean formulas~\cite{bool0a,bool0b,W97b} and the
counting of finite models~\cite{W97a}. The topic was resumed later, studying those random Boolean
formulas under different aspects~\cite{bool1,bool2a,bool2b}, analogous formulas for other
logical models~\cite{tautcount,bool3,bool4}, studying the number of tautologies~\cite{truth} or
the number of proofs in propositional logic~\cite{proofcount}, comparing logical
systems~\cite{GK12}, comparing size notions of Boolean trees~\cite{boolsize}, or applying results
in that domain to satisfiability~\cite{sat1}. In \cite{sat2} combinatorial enumeration of graphs
was applied to study satisfiability problem and some generalizations.

The to our knowledge first enumerative investigation of lambda terms was performed in
\cite{DGKRTZ10}. Later, particular classes of lambda terms like linear and affine terms have been
enumerated~\cite{bcibck,bci_in_bck}. The random generation of terms was for instance treated
in~\cite{bcibck,GL13}. Parts of these results have been extended to more general classes of lambda
terms~\cite{bci_gen}. Lambda terms related to combinatory logic were studied
in~\cite{BGZ15,tromp:DSP:2006:628} where \cite{BGZ15} also investigates the question on how many
normalizing lambda term there are among all terms. As opposed to the above mentioned results,
where lambda terms were interpreted as graphs and their size is then the number of vertices of
their graph representation, \cite{tromp:DSP:2006:628} introduced a size concept related to the
De Bruijn representation of lamda terms. The resulting enumeration problem is then different and
also simpler when approaching it with generating functions. The reason is there are much fewer
terms of a given size than in the graph interpretation, making the generating function analytic
around the origin and therefore amenable to analytic methods. Several variation of this model
have been studied, see~\cite{BGLZ16,GLBinLT,GL13,Le13}.

Based on numerical experiments Grygiel and Lescanne~\cite{GLBinLT,GL13} conjectured that the
growth rate of the counting sequence is smaller than it is for typical tree-like structures,
which is of the form $\alpha^n n^{-3/2}$ for some positive $\alpha$. In~\cite{GG16}, however, a
lower and an upper bound for the numbers were derived which disproved this conjecture. These two
bounds were derived for the whole class of models we defined via the general size notion above
and for concrete models like Tromp's~\cite{tromp:DSP:2006:628} binary lambda terms, they proved
to be of the same order of magnitude, being of the form $\alpha^n n^{-3/2}$, and very close to
each other:  Indeed the difference of the multiplicative constants is less than $10^{-7}$. This
suggests that in fact asymptotic equivalence to $C \alpha^n n^{-3/2}$ for some constant $C$ holds,
though theoretically there could be very small oscillations. Examples of periodic oscillations
are frequently found in number theory and also combinatorial problems some exhibit surprising
periodicities, for example the mathematical puzzle of group Russian roulette described by
Winkler~\cite{W04} and recently analyzed by van de Brug {\em et al.}~\cite{BKM17}. Further
examples, including some with very small oscillations, were discussed in \cite{Pr06}.

\subsubsection*{Notations.} 
We introduce some notations which will be frequently used throughout
the paper: If $p$ is a polynomial, then $\mathrm{RootOf}\left\{p\right\}$ will denote the
smallest positive root of $p$. Moreover, we will write $[z^n]f(z)$ for the $n$th coefficient of
the power series expansion of $f(z)$ at $z=0$ and $f(z) \prec g(z)$ (or $f(z) \preceq g(z)$) to
denote that $\left[z^n\right] f(z) < \left[z^n\right] g(z)$ (or $\left[z^n\right] f(z) \leq
\left[z^n\right] g(z)$) for all integers $n$. 

\subsubsection*{Plan of the paper and results.} 
The first aim of this paper is the asymptotic enumeration of closed lambda terms of given size, as
the size tends to infinity. In the next section we define several classes of lambda terms
as well as the generating function associated with them and present the enumeration result for
those classes. We derive the asymptotic equivalent of the number of closed terms of given size up
to a constant factor. 

Having solved the basic enumeration problem one can ask then for further parameters. 
As an example, we pick the \emph{number of abstractions in a term} in Section~\ref{q_abstractions}
and compute the asymptotic number of lambda terms containing an \emph{a priori} fixed number of
abstractions as well as the asymptotic number of terms with a bounded number of abstractions. 

Terms with are not reducible by so-called $\beta$-reduction (=terms in normal form) play an
important role in lambda calculus. In Section~\ref{normal} we show that asymptotically almost no
term is in normal form and quantify the decay of the fraction of terms in normal form. In fact,
the fraction decays exponentially fast. 

The final Section~\ref{Boltzmann} is devoted to the uniform random generation of lambda terms. We 
exploit the fast convergence in our asymptotic results to construct a Boltzmann sampler which is
based on an adapted rejection method. This sampler is as efficient as the samplers for generating
trees and outranges all existing samplers for lambda terms. On a standard laptop terms of size more
than one million can easily be generated in a couple of minutes.

\section{The asymptotic number of lambda terms}
In order to count lambda terms of a given size we set up a formal equation which is then
translated into a functional equation for generating functions. For this we will utilise the 
symbolic method developed in~\cite{Flajolet:2009:AC:1506267}. 

Let us introduce the following atomic classes, \emph{i.e.}, they consist of one single element: the class of zeros $\cZ$, the class of successors
$\cS$, the class of abstractions $\cU$ and the class of applications $\cA$. 
Then the class $\cL_{\infty}$ of lambda terms can be described as follows:
\begin{equation}\label{mr:eq:0}
\cL_{\infty} = \seq (\cS) \times \cZ \;\cup\; \cU  \times \cL_{\infty} \;\cup\; \cA \times
\cL_{\infty}^2.
\end{equation}
This specification of the set of all lambda terms can be seen as follows: A lambda term is either
a De Bruijn index, which is a sequence of symbols $S$ followed by a zero, or a pair built from an
abstraction (= the element of $\cU$) and a lambda term or two lambda terms concatenated by the 
application symbol between them, where we write them as a triple made of the application symbol in
the first component and the two lambda terms in the second and third component.

The number of lambda terms of size $n$, denoted by $L_{\infty,n}$, is $\left|\left\{t \in
\cL_{\infty}: |t|=n \right\}\right|$.
Let $L_{\infty}(z) = \sum_{n \geq 0} L_{\infty,n} z^n$ be the generating function associated with
$\cL_{\infty}$. 
Then specification~\eqref{mr:eq:0} gives rise to a functional equation for the generating function $L_{\infty}(z)$:
\begin{equation}\label{mr:eq:1}
L_{\infty}(z) = z^a \sum_{j=0}^{\infty} z^{bj} + z^c L_{\infty}(z) + z^d L_{\infty}(z)^2.
\end{equation}
Solving \eqref{mr:eq:1} we get
\begin{equation}\label{mr:eq:2}
L_{\infty}(z) = \frac{1 - z^c - \sqrt{(1-z^c)^2 - \frac{4 z^{a+d}}{1-z^b}}}{2 z^d}, 
\end{equation}
which defines an analytic function in a neighbourhood of $z=0$. 

\begin{proposition}\label{i:prop:1}
Let $\rho$ be the smallest positive root of $(1-z^b)(1-z^c)^2 - 4 z^{a+d}$. Then, as $z\to\rho$ in
such a way that $\arg(z-\rho)\neq 0$, the function $L_{\infty}(z)$ admits the local expansion 
\begin{equation}\label{mr:eq:3}
L_{\infty}(z) = a_{\infty} + b_{\infty} \left(1 - \frac{z}{\rho}\right)^{\frac12} + \BigO{\left|1
- \frac{z}{\rho}\right|} 
\end{equation}
for some constants $a_{\infty} > 0, b_{\infty} < 0$ that depend on $a,b,c,d$.
\end{proposition}
\begin{proof}
Let $f(z) = (1-z^b)(1-z^c)^2 - 4 z^{a+d}$. Then $\rho$ is the smallest positive solution of $f(z)
= 0$. If we compute derivative $f'(z) = -4 (a + d) z^{a+d-1} - 2 c z^{c-1} (1 - z^b) (1 - z^c)
- b z^{b-1} (1 - z^c)^2$ we can observe that all three terms are negative for $0 < z < 1$.  Since
$0 < \rho < 1$, the function $f(z)$ does not have a double root at $\rho$ and thus $L_{\infty}(z)$
has an algebraic singularity of type $\frac12$ which means that its Newton-Puiseux expansion is
of the form~(\ref{mr:eq:3}).

Since $L_{\infty}(z)$ is a power series with positive coefficients, we know that $a_{\infty} =
L_{\infty}(\rho) > 0$ and $b_{\infty} < 0$. 
\end{proof}

\begin{corollary}
The coefficients of $L_\infty(z)$ satisfy $[z^n]L_\infty(z)\sim C\rho^{-n}n^{-3/2}$, as $\nti$,
where $C=-b_\infty /(2\sqrt\pi)$. 
\end{corollary}

\begin{remark}
The two constants from Proposition~\ref{i:prop:1} are 
\[
a_\infty=\frac{1-\rho^c}{2\rho^d} \qquad\text{ and }\qquad
b_\infty=\frac{\sqrt{4 (a + d) \rho^{a+d} - 2 c \rho^c (1 - \rho^b) (1 - \rho^c)
- b \rho^b (1 - \rho^c)^2}}{2\rho^d\sqrt{1-\rho^b}}.
\]
\end{remark}

Let us define the class of $m$-open lambda terms, denoted $\cL_m$, as
\[
\cL_{m} = \left\{t \in \cL_{\infty}: \textrm{a prefix of } m \textrm{ abstractions } \lambda
\textrm{ makes $t$ a closed term}\right\}.
\]
We remark that any $m$-open lambda term is obviously $(m+1)$-open as well. 
The number of $m$-open lambda terms of size $n$ is denoted by $L_{m, n}$ and the generating
function associated with the class by $L_{m}(z) = \sum_{n \geq 0} L_{m,n} z^n$. 
Similarly to specification \eqref{mr:eq:0} for $\cL_{\infty}$, the class $\cL_{m}$ can be
specified as
\begin{equation} \label{m_open_def}
\cL_m = \seq_{\le m-1} (\cS) \times \cZ \;\cup\; \cU  \times \cL_{m+1} \;\cup\; \cA \times
\cL_m^2, 
\end{equation} 
and this specification yields the functional equation 
\begin{equation}\label{mr:eq:5}
L_{m}(z) = z^a \sum_{j=0}^{m-1} z^{bj} + z^c L_{m+1}(z) + z^d L_{m}(z)^2
\end{equation}
for the associated generating function. 
Note that $L_0(z)$ is the generating function of the set $\mathcal{L}_0$ of all 
closed lambda terms.

Now, we are able to state one of the main results of the paper:

\begin{theorem}\label{thm:1}
Let $\rho$ be the smallest positive root of $(1-z^b)(1-z^c)^2 - 4 z^{a+d}$. Then there exists a 
positive constant $C$ (depending on $a, b, c, d$ and $m$) such that the number of $m$-open lambda terms of size $n$ satisfies
\begin{equation} \label{mainresult}
\left[z^n\right] L_{m}(z) \sim C n^{-\frac32} \rho^{-n}, \text{ as } \nti. 
\end{equation}
\end{theorem}

\begin{remark} 
In case of given $a, b, c, d$ and $m$ it is possible to compute $C$ numerically.
Indeed, in~\cite{GG16} an approximation procedure has been developped to construct a subset and
a superset of $\cL_m$. Using this, we showed that in one approximation step in the case of natural
counting ($a=b=c=d=1$) the constant determining the asymptotic number of closed lambda terms
($m=0$) we have $0.07790995266\le C_{1,1,1,1,0} \le 0.0779099823$,
and for Tromp's binary lambda terms we obtained $0.01252417\le C_{2,1,2,2,0} \le 0.01254594$. So,
the fast convergence of the approximation procedure allows us to determine the constant very
accurately with only moderate computational effort. 
\end{remark}

\begin{remark}
The generating function $L_{m}(z)$ can be expressed as an infinitely nested radical. Solving 
\eqref{mr:eq:5}, a quadratic equation for $L_m(z)$ gives an expression involving a radical the
radicand of which contains $L_{m+1}(z)$. Then iterate to express $L_{m+1}(z), L_{m+2}(z),\dots$ in
the same way. Eventually, we get 
\begin{equation} \label{infnest}
L_{m}(z) = 
\frac{
1 - \sqrt{w_m(z) + 2z^c \sqrt{w_{m+1}(z) + 2z^c \sqrt{w_{m+2}(z) + 2z^c \sqrt{\cdots}}}}}{2z^d}
\end{equation}
where
\[ 
w_{j}(z) = 1 - 4z^{a+d} \frac{1 - z^{j b}}{1-z^b} - 2z^c \text{ for } j=0,1,2,\dots. 
\] 
Nested radicals appeared in other enumeration problems concerning lamdba terms as well. Some
examples of finitely nested radicals and divergent infinitely nested radicals were studied in
\cite{BGGG17}. The infinitely nested radical \eqref{infnest} is convergent for $|z|\le\rho$, in
contrast to the one considered in \cite{BGGG17}. Though it seems still hard to study it directly,
we eventually succeeded in showing Theorem~\ref{thm:1} by approximating it by a finite nested
radical and then studying the convergence to the infinitely nested one. 
\end{remark}

The proof of Theorem~\ref{thm:1} works in two steps. The asymptotics of $\left[z^n\right]
L_{m}(z)$ heavily depends on the behaviour of the function $L_m(z)$ near its smallest positive
singularity. By the transfer theorems of Flajolet and
Odlyzko~\cite{DBLP:journals/siamdm/FlajoletO90} we know that the 
location of the singularity gives the exponential growth rate, the nature of the singularity
determines the subexponential factor. The first task is therefore to locate the dominant
singularities of $L_m(z)$, for all $m\in \mathbb N$. 
We will show that all $L_m(z)$ have the same dominant singularity and that it is equal to 
$\rho$, the dominant singularity of $L_\infty(z)$. The second step is then a more
detailed analysis in order to show that all function have a square-root type singularity at $\rho$.

Let $\cK_m = \cL_{\infty} \setminus \cL_{m}$ and $K_{m}(z) = L_{\infty}(z) - L_{m}(z)$. Then using
\eqref{mr:eq:1}~and~\eqref{mr:eq:5} we obtain 
\begin{equation}\label{mr:eq:6}
K_{m}(z) = z^a \sum_{j=m}^{\infty} z^{bj} + z^c K_{m+1}(z) + z^d K_{m}(z) L_{\infty}(z) + z^d K_{m}(z) L_{m}(z).
\end{equation}
which implies 
\begin{equation}\label{mr:eq:7}
K_{m}(z) = \frac{z^{a+bm}}{(1-z^b)(1 - z^d (L_{\infty}(z) + L_{m}(z)))} + \frac{z^c}{1 - z^d (L_{\infty}(z) + L_{m}(z))}{} K_{m+1}(z).
\end{equation}
Note that $K_m(z)$ as well as $L_m(z)$ define analytic functions in a neighbourhood of $z=0$. 

We introduce the
class $\cL_{m}^{(h)}$ of lambda terms in $\cL_{m}$ where the length of each
string of successors is bounded by a constant integer $h$. As before, set $L_{m,n}^{(h)} =
\left|\left\{t \in \cL_{m}^{(h)}: |t|=n \right\}\right|$ and $L_{m}^{(h)}(z) = \sum_{n \geq 0}
L_{m,n}^{(h)} z^n$. Then $L_{m}^{(h)}(z)$ satisfies the functional equation
\begin{equation}\label{mr:eq:9}
L_{m}^{(h)}(z) = 
\begin{cases} 
z^a \sum_{j=0}^{m-1} z^{bj} + z^c L_{m+1}^{(h)}(z) + z^d L_{m}^{(h)}(z)^2 & \textrm{if } m < h,
\\
z^a \sum_{j=0}^{h-1} z^{bj} + z^c L_{h}^{(h)}(z) + z^d L_{h}^{(h)}(z)^2 & \textrm{if } m \geq h.
\end{cases}
\end{equation}
Notice that for $m \geq h$ we have a quadratic equation for $L_m^{(h)}(z)=L_{h}^{(h)}(z)$ 
that has the solution
\[
L_{h}^{(h)}(z) = \frac{1-z^c-\sqrt{(1-z^c)^2 - 4 z^{a+d} \frac{1-z^{b h}}{1-z^b}}}{2 z^d}.
\]
For $m < h$ we have a relation between $L_{m}^{(h)}(z)$ and $L_{m+1}^{(h)}(z)$ which gives rise to a
representation of $L_m^{(h)}(z)$ in terms of a nested radical
(\emph{cf.}~\cite{BGGG17}) after all. Indeed, for $m < h$ we have
\begin{equation}\label{nestedrad}
L_{m}^{(h)}(z) = 
\frac{
1 - \sqrt{r_m(z) + 2z^c \sqrt{r_{m+1}(z) + 2z^c \sqrt{\cdots \sqrt{r_{h-1}(z) + 2z^c \sqrt{r_h(z)}}}}}
}
{2z^d}
\end{equation}
where
\begin{equation}\label{radic}
r_{j}(z) = 
\begin{cases}
1 - 4z^{a+d} \frac{1 - z^{j b}}{1-z^b} - 2z^c & \textrm{if } m \leq j < h-1, \\
1 - 4z^{a+d} \frac{1-z^{(h-1)b}}{1-z^b} - 2z^c + 2z^{2c} & \textrm{if } j = h-1, \\
(1-z^c)^2 - 4 z^{a+d} \frac{1-z^{b h}}{1-z^b} & \textrm{if } j = h.
\end{cases}
\end{equation}

\begin{lemma}\label{conv}
For all $m\ge0$ the dominant singularity $\rho^{(h)}:=\rho_{m}^{(h)}$ of $L_m^{(h)}(z)$ is
independent of $m$. Moreover, we have $\lim_{h \to \infty} \rho^{(h)} = \rho$. 
\end{lemma}

\begin{proof} 
Since $\cL_{m}^{(h)}\subseteq \cL_\infty$, the dominant singularity of $L_m^{(h)}(z)$ cannot be
larger than $\rho$. Furthermore, it must be the smallest positive number where any of the
radicands in \eqref{nestedrad} becomes zero. Since $\rho<1$, one easily sees from \eqref{radic}
that $\rho_{(h)}$ is the smallest positive 
root of $r_h(z)$ and that it is of type $\frac12$, \emph{i.e.}, there are constants
$a_{m}^{(h)}, b_{m}^{(h)}$ depending on $m$ and $h$ such that 
$L_m^{(h)}(z) = a_{m}^{(h)} + b_{m}^{(h)} \left(1 - \frac{z}{\rho^{(h)}}\right)^{\frac12} +
\BigO{\left|1 - \frac{z}{\rho^{(h)}}\right|}$, as $z \to \rho^{(h)}$ in such a way that
$\arg\(z-\rho^{(h)}\)\neq 0$. Since $r_h(z)$ is independent of $m$, so is $\rho_{m}^{(h)}$. 
Notice that in the unit interval $r_h(z)$ converges uniformly to $r(z)$, the radicand in
\eqref{mr:eq:2}. Thus $\lim_{h \to \infty} \rho^{(h)} = \rho$ since $\rho$ is the smallest positive
root of $r(z)$.
\end{proof}

Let us begin with computing the radii of convergence of the functions $K_m(z)$ and $L_m(z)$. For
the case of binary lambda calculus Lemmas~\ref{l:01}~and~\ref{l:1} were already proven 
in~\cite{GLBinLT}. To extend those results to our more general setting, we will use different
techniques. 

\begin{lemma}\label{l:0}
For all $m \geq 0$ the radius of convergence of $K_{m}(z)$ equals $\rho$ (the radius of
convergence of $L_{\infty}(z)$).
\end{lemma}

\begin{proof}
Inspecting (\ref{mr:eq:7}) reveals that $K_m(z)$ can only be singular if $K_{m+1}(z)$, $L_m(z)$,
or $L_\infty(z)$ is singular or any of the expressions appearing in some denominator vanish. Both 
$\cL_{m}$ and $\cK_m$ are subsets of $\cL_\infty$, thus the dominant singularity of $K_m(z)$
cannot be larger than $\rho$. Since $\rho<1$, the term $1-z^b$ is certainly positive. So the only
subexpression which could cause a singularity smaller than $\rho$ is 
$\frac{1}{1 - z^d (L_{\infty}(z) + L_{m}(z))}$. 
This is the generating function of a sequence of combinatorial structures associated
with the generating function $z^d (L_{\infty}(z) + L_{m}(z))$. 
One can check that we are not in the case of a supercritical sequence schema, \emph{i.e.}, 
the dominant singularity of the considered fraction does not come from a root of its denominator,
see~\cite[pp.~293]{Flajolet:2009:AC:1506267}), because $1 - \rho^d \left(L_{\infty}(\rho) +
L_m(\rho)\right) > 0$. This follows from 
\[
\rho^d \left(L_{\infty}(\rho) + L_m(\rho)\right) \leq 2 \rho^d L_{\infty}(\rho) = 1 - \rho^c < 1.
\]
The first inequality holds because $L_{\infty}(\rho) \geq L_m(\rho)$ for all $m\ge 0$ and the
second one because $\rho > 0$. 
Therefore, we conclude after all that for all $m \geq 0$ the radius of convergence of $K_{m}(z)$ equals $\rho$, the radius of convergence of $L_{\infty}(z)$. 
\end{proof}

Note that we are not done yet. Indeed, the singularities (at $z=\rho$) of the functions on the
right-hand side of $L_m(z)=L_\infty(z)-K_m(z)$ might cancel and $L_m(z)$ would have a larger
dominant singularity. So, we first show that all $L_m(z)$ have the same dominant singularity and
in a second step that this singularity is indeed at $\rho$. 

\begin{lemma}\label{l:01}
All the functions $L_{m}(z)$, $m \geq 0$, have the same radius of convergence.
\end{lemma}

\begin{proof}
Let $\rho_m$ denote the radius of convergence of the function $L_{m}(z)$. 
From the definition of $L_{m}(z)$ it is known that for all $m \geq 0$ and for all $n$
we have $\left[z^n\right] L_{m}(z) \leq \left[z^n\right] L_{m+1}(z)$ and therefore $\rho_m \geq
\rho_{m+1}$. 
Moreover, from (\ref{mr:eq:5}) we know
\[
L_{m+1}(z) = - z^{a-c} \sum_{j=0}^{m-1} z^{bj} + z^{-c} L_{m}(z) - z^{d-c} L_{m}(z)^2.
\]
Hence, $L_{m+1}(z)$ is singular whenever $L_{m}(z) - z^d L_{m}(z)^2$, since the seeming singularity
at 0 must cancel, as know that $L_{m+1}(z)$ is regular there. Obviously, the radius of convergence
of $L_{m}(z) - z^d L_{m}(z)^2$ is at least $\rho_m$ and so $\rho_{m} \leq \rho_{m+1}$. 
\end{proof}

\begin{lemma}\label{l:1}
For all $m \geq 0$ the radius of convergence of $L_{m}(z)$ equals $\rho$.
\end{lemma}

\begin{proof}
Take $L_{m}^{(m)}$, defined in~\eqref{mr:eq:9}.
Recall that $\rho^{(m)}, \rho_m$ and $\rho$ denote the radii of convergence of $L_{m}^{(m)}(z),
L_{m}(z)$ and $L_{\infty}(z)$, respectively.
Notice that for all $m, n \geq 0$ we have $\left[z^n\right] L_{m}^{(m)}(z) \leq \left[z^n\right]
L_{m}(z) \leq \left[z^n\right] L_{\infty}(z) $ and thus $\rho^{(m)} \geq \rho_m \geq \rho$. 
Now, the assertion follows from Lemmas~\ref{conv} and~\ref{l:01}. 
\end{proof}

In order to prove Theorem~\ref{thm:1} we have to show that all functions $L_m(z)$ satisfy
$L_m(z)\sim a_m-b_m\sqrt{1-\frac z\rho}$, as $z\to\rho$, for some suitable constants $a_m$ and
$b_m$. The idea is as follows: We know that $L_\infty(z)$ admits such a representation (\emph{cf.}
Proposition~\ref{i:prop:1}) and that the functions $L_m(z)$ are related to each other via the
recurrence \eqref{mr:eq:5}. However, we cannot start from $L_\infty(z)$ and then trace backwards. 
But if $N$ is large, we then we expect that $L_N(z)$ is close to $L_\infty(z)$. So, if we replace
$L_N(z)$ by $L_\infty(z)$, thus having the desired local shape at index $N$, and then trace
backwards to index $0$, then we will obtain the desired shape for all functions up to index $N$.
Of course, the functions we obtain are not precisely the $L_m(z)$, but close to them and if we can
control the error, we are done. 

So, define functions $L_{m,N}(z)$ by 
\begin{align} 
L_{N,N}(z)&=L_\infty(z), \nonumber \\
L_{m,N}(z)&=z^a\sum_{j=0}^{m-1} z^{bj} + z^c L_{m+1,N}(z)+z^dL_{m,N}(z)^2. \label{rec:L_mN}
\end{align} 
These functions can be interpreted in a straight-forward way as generating functions associated to
some combinatorial classes. Let $\mathcal L_{m,N}$ denote the combinatorial class with generating
function $L_{m,N}(z)$, where $m=0,1,\dots,N$. 

Let us define the distance between two power series by 
$$
d \( \sum_{n\ge 0} a_nz^n, \sum_{n\ge 0} b_nz^n \) :=
\begin{cases}
0 & \text{if } a_n=b_n \text{ for all } n,  \\
2^{-\min\{\ell\,\mid\,a_\ell\neq b_\ell\}} & \text{otherwise}.
\end{cases}
$$
We remark that this distance makes the ring of formal power series of the reals a complete metric
space. 

Obviously, we have $\mathcal L_N\subseteq \mathcal L_{N,N}$, and the two classes differ only in
the set of $(N+1)$-open terms. Thus $d(L_N(z),L_{N,N}(z))<2^{-N}$. By relation \eqref{rec:L_mN}
this property propagates to smaller indices, \emph{i.e.}, $d(L_m(z),L_{m,N}(z))<2^{-N}$ holds for
$m=0,1,\dots, N$.

Moreover, observe how $\mathcal L_{m,N}$ is constructed. The terms in $\mathcal L_{m,N}$ fall
into one of three categories. The first two are: terms of the form $S^k0$ with $k<m$ and terms
constructed of two terms from $\mathcal L_{m,N}$ which are connected via an application. The last
category is an abstraction followed by a term from $\mathcal L_{m+1,N}$, \emph{i.e.}, after an
abstraction we go one level up. The latter is done as long as level $N$ is reached. After a
further abstraction any term in $\mathcal L_\infty$ may follow (contrary to the construction
of $\mathcal L_m$, where a term from $\mathcal L_{N+1}$ must follow in that case).  Therefore, we
have $\mathcal L_{m,N} \supseteq \mathcal L_m$. 

\begin{lemma}\label{speedlemma}
For all $m\ge 0$ we have $\lim_{N\to\infty} L_{m,N}(z)= L_m (z)$ uniformly for $|z|\le \rho$.
Moreover, the speed of convergence can be estimated by $|L_{m,N}(z)-L_m (z)|=O\(N^{-1/2}\)$, as
$N\to\infty$ and uniformly in $m$, i.e., the implicit $O$-constant neither depends on $N$ nor on
$m$. 
\end{lemma}

\begin{proof}
Since the above described relaxation takes place only at the levels $N,N+1,N+2\dots$, 
we also obtain the inclusion
$\mathcal L_{m,N}\supseteq \mathcal L_{m,N+1}$. Thus the sequence $(L_{m,N}(z))_{N=m,m+1,\dots}$
is decreasing (seen as sequence of functions on $[0,\rho]$ as well as coefficient-wise). Hence,
$\lim_{N\to\infty} L_{m,N}(z)$ exists for any $0\le z\le \rho$. To find the limit, observe that 
\begin{equation} \label{difference_of_Ls}
L_{m,N}(z)=\sum_{n\ge 1} z^n [z^n] L_m (z) + \sum_{n>N} F_n z^n, 
\end{equation}  
and $F_n$ is the number of terms in $\mathcal L_{m,N}\setminus \mathcal L_m$. Therefore $F_n\ge
0$. Moreover, since $\mathcal L_{m,N}\setminus \mathcal L_m \subseteq \mathcal L_\infty$, for any
$z$ satisfying $|z|\le\rho$              
the series $\sum_{n>N} F_n z^n$ is bounded from above by the
remainder $R:=\sum_{n>N} \rho^n [z^n] L_\infty (z)$ of the convergent series for $L_\infty (z)$.
Thus we can choose $N$ so that $\sum_{n>N} F_n z^n$ is arbitrarily small. As $R$ is independent
of $z$, the convergence is uniform. 

The bound for the convergence rate follows from $F_n< [z^n] L_\infty (z) \sim C\rho^{-n}n^{-3/2}$,
which holds for all $n\le N$ and is independent of $m$. 
\end{proof}

\begin{lemma}\label{l:02}
For any $N>0$ all the functions $L_{m,N}(z)$, $m=0,1,\dots,N$, have radius of convergence equal to
$\rho$. 
\end{lemma}

\begin{proof}
By \eqref{rec:L_mN} we have 
\begin{equation} \label{LmN_explicit}
L_{m,N}(z)=\frac{1-\sqrt{1-4z^{a+d}\frac{1-z^{bm}}{1-z^b}-4z^{c+d} L_{m+1,N}(z)}}{2z^d}.
\end{equation} 
Since $L_{m,N}(z)$ has eventually positive coefficients, it has only one singularity on its circle
of convergence and this must lie on the positive real axis. Moreover, since $L_{N,N}(z)$ is 
singular at $\rho$, the singularity of $L_{m,N}(z)$ cannot be larger than $\rho$. 
In order to show that it is exactly $\rho$, observes that the radicand in \eqref{LmN_explicit} is
decreasing for $z>0$. Thus it suffices to show that the radicand is positive at $z=\rho$.

Note that the definition of $\rho$ implies $4\rho^{a+d}/(1-\rho^b)=(1-\rho^c)^2$ and that
$L_{m+1,N}(\rho)\le L_\infty(\rho)$. Finally, by \eqref{mr:eq:2} we have
$\rho^dL_\infty(\rho)=(1-\rho^c)/2$. Thus 
\begin{align*} 
4\rho^{a+d}\frac{1-\rho^{bm}}{1-\rho^b}+4\rho^{c+d}L_{m+1,N}(\rho) & \le  
(1-\rho^{bm})(1-\rho^c)^2 + 2\rho^c (1-\rho^c) \\
& \le (1-\rho^c)(1+\rho^c)=(1-\rho^{2c}) < 1. 
\end{align*} 
Here we used $1-\rho^{bm}<1$ (because of $b\ge 1$) in the penultimate step and $c\ge 1$ in the
last step. This implies positivity of the radicand at $z=\rho$. 
\end{proof}

\begin{corollary}
For all intergers $N>0$ and $0\le m\le N$ there are positive constants $a_{m,N}$ and $b_{m,N}$ such that 
$$
L_{m,N}(z) \sim a_{m,N}-b_{m,N}\sqrt{1-\frac z\rho}, 
$$
as $z\to \rho$. 
\end{corollary}

\begin{lemma}
For every $m\ge 0$ the sequences $(a_{m,N})_{N\ge m}$ and $(b_{m,N})_{N\ge m}$ are convergent. 
\end{lemma}

\begin{proof}
We know that $a_{m,N}=L_{m,N}(\rho)$ and $\lim_{N\to\infty} L_{m,N}(z)= L_m (z)$. Thus
$(a_{m,N})_{N\ge m}$ converges. 

Recall that $(L_{m,N}(z))_{N\ge m}$ is (coefficient-wise) decreasing; thus $(b_{m,N})_{N\ge m}$ is
increasing. Moreover, $b_{m,N}\le b_\infty$; thus $(b_{m,N})_{N\ge m}$ converges as well. 
\end{proof}

The following lemma completes the proof of Theorem~\ref{thm:1}
\begin{lemma}
Let $a_m:=\lim_{N\to\infty} a_{m,N}$ and $b_m:=\lim_{N\to\infty} b_{m,N}$. Then, as $z\to\rho$,  
$$
L_m\sim a_m-b_m \sqrt{1-\frac z\rho}.
$$
\end{lemma}

\begin{proof}
We know already that for each $\eps>0$ and all $z$ sufficiently close to $\rho$ we have 
$$
\left|\frac{L_{m,N}(z)}{a_{m,N}-b_{m,N}\sqrt{1-\frac z\rho}} -1 \right|<\eps.
$$
By \eqref{difference_of_Ls} we know that $L_{m,N}(z)$ converges uniformly to $L_m (z)$. Thus we
can choose a sufficiently large $N$ such that $L_{m,N}(z)$, $a_{m,N}$ and $b_{m,N}$ are
arbitrarily close to $L_m (z)$, $a_m$ and $b_m$, respectively, and we are done.
\end{proof}


\section{Enumeration of lambda terms with prescribed number of abstractions}
\label{q_abstractions}

\subsection{Lambda terms containing $q$ abstractions} 

We consider the class of $m$-open lambda terms with exactly $q$ abstractions, denoted
$\mathcal{L}_{m,q}$. Note that $\mathcal{L}_{0,q}$ is then the class of closed lambda terms with exactly $q$ abstractions.

The number of $m$-open lambda terms of size $n$ with exactly $q$ abstractions is denoted by
$L_{m,q,n}$. As before, we define the generating function associated with the class
$\mathcal{L}_{m,q}$ by $L_{m,q}(z) = \sum_{n \geq 0} L_{m,q,n} z^n$.

We shall set up recurrence relations for the generating functions $L_{m,q}(z)$.
The objects in $\mathcal{L}_{m,0}$ are binary plane trees with sequences of successors of length
less than $m$ followed by a zero attached as leaves. Therefore $L_{0,0}(z) = 0$ 
(because without any abstraction the term cannot be closed) and for $m>0$
\[
L_{m,0}(z) = z^a \sum_{j=0}^{m-1} z^{b j} + z^d L_{m,0}(z)^{2}
\]
what can be solved as
\[
L_{m,0}(z) = \frac{1 - \sqrt{1-4 z^{a+d} \sum_{j=0}^{m-1} z^{b j}}}{2 z^d}.
\]
For general $q>0$, a term has either an abstraction as the root and $q-1$ abstractions below or an
application as the root, and the $q$ abstractions are distributed into $l$ being in the left
subterm, and $q-l$ being in the right subterm. Hence we obtain
\[
L_{m,q}(z) = z^c L_{m+1,q-1}(z) + z^d \sum_{l=0}^{q} L_{m,l}(z) L_{m,q-l}(z).
\]
From this equation we can easily derive an equation for $L_{m,q}(z)$ in terms of $L_{m+1,q-1}$ and
$L_{m,l}$ for $l<q$. We get 
\begin{equation}\label{zq:eq:1}
L_{m,q}(z) = \frac{1}{\sqrt{1-4 z^{a+d} \sum_{j=0}^{m-1} z^{b j}}} \left(z^c L_{m+1,q-1}(z) + z^d \sum_{l=1}^{q-1} L_{m,l}(z) L_{m,q-l}(z)\right).
\end{equation}
The number of closed lambda terms with exactly $q$ abstractions, which we are mainly interested in, is then described by
\[
L_{0,q}(z) = z^c L_{1,q-1}(z) + z^d \sum_{l=1}^{q-1} L_{0,l}(z) L_{0,q-l}(z).
\]

\begin{lemma}
Let $\delta_{m}(z) = \sqrt{1 - 4 z^{a+d} \sum_{j=0}^{m-1} z^{b j}}$. Then, for all $m, q \geq 0$,
there exists a rational function $R_{m,q}(z)$ such that 
\begin{equation}\label{zq:eq:2}
L_{m,q}(z) = - \frac{z^{c q} \delta_{m+q}(z)}{2 z^d \prod_{i=0}^{q-1} \delta_{m+i}(z)} + R_{m,q}(z).
\end{equation}
Moreover, the denominator of $R_{m,q}(z)$ is of the form $\prod_{i=0}^{q-1} \delta_{m+i}(z)^{\alpha_{i}}$ where the exponents $\alpha_{0}, \ldots, \alpha_{q-1}$ are positive integers.
\end{lemma}
\begin{proof}
The proof is based on induction on $q$. To start the induction observe that for all $m \geq 0$ we have $L_{m,0}(z) = - \frac{\delta_m(z)}{2 z^d} + \frac{1}{2 z^d}$, so $R_{m,0}(z) = \frac{1}{2 z^d}$.
Now, assume that~(\ref{zq:eq:2}) is true for $L_{m,0}(z), \ldots, L_{m,q}(z)$ for all $m \geq 0$. Then by~(\ref{zq:eq:1}) we have
\[
L_{m,q+1}(z) = \frac{1}{\delta_{m}(z)} \left(-\frac{z^{c (q+1)} \delta_{m+1+q}(z)}{2 z^d \prod_{i=0}^{q-1} \delta_{m+1+i}(z)} + z^c R_{m+1,q}(z) + z^d \sum_{l=1}^{q} L_{m,l}(z) L_{m,q+1-l}(z)\right).
\]
The induction hypothesis implies that each $L_{m,l}(z)$ is itself a rational function of $z, \delta_{m}(z),\ldots,\delta_{m+l-1}(z)$. Hence, by setting 
\begin{equation}\label{zq:eq:3}
R_{m,q+1}(z) = \frac{1}{\delta_{m}(z)} \left(z^c R_{m+1,q}(z) + z^d \sum_{l=1}^{q} L_{m,l}(z) L_{m,q+1-l}(z)\right)
\end{equation}
we obtain
\[
L_{m,q+1}(z) = -\frac{z^{c (q+1)} \delta_{m+1+q}(z)}{2 z^d \prod_{i=0}^{q} \delta_{m+i}(z)} + R_{m,q+1}(z).
\]
The expression in the denominator of the $R_{m,q}(z)$ comes readily from \eqref{zq:eq:2} and the
recurrence relation~\eqref{zq:eq:3}.
\end{proof}

Now, a singularity analysis of \eqref{zq:eq:2} leads to the desired asymptotics for the number of
lambda terms with $q$ abstractions. 

\begin{lemma}
Let $\xi_{m+q}$ denote the smallest positive root of $1 - 4 z^{a+d} \sum_{j=0}^{m+q-1} z^{b j}$.
Then the number of $m$-open lambda terms with exactly $q$ abstractions and size $n$ is $0$ if
$m,q=0$ or if $n$ is any positive integer that cannot be expressed in the form $(y+1)a+xb+qc+yd$
with nonnegative integers $x,y$ satisfying $x\le (m-1+q)(y+1)$; 
otherwise its asymptotic value is
\begin{equation} \label{q_abs}
L_{m,q,n} \sim 
C \frac{\xi_{m+q}^{-n}}{2 \sqrt{\pi n^3}},  \quad \mathrm{as} \quad n \to \infty,
\end{equation} 
where
\[
C = \frac{\xi_{m+q}^{c q - d} \sqrt{\xi_{m+q}^{a+d} \sum_{j=0}^{m+q-1} (a+d+bj) \xi_{m+q}^{b j}}}{\prod_{i=0}^{q-1} \delta_{m+i}(\xi_{m+q})}.
\]
\end{lemma}

\begin{proof}
The restriction on the expressability of $n$ as a combination of $a,b,c,d$ results from the fact
that the terms have to have exactly $q$ abstractions and must be $m$-open. If there are $y$
applications then we have $y+1$ leaves, each of which needs at most $q$ abstractions to get
$m$-open. Thus the number of $S$'s in each leaf is bounded by $m-1+q$. So the total number $x$ of
$S$'s is bounded by $(m-1+q)(y+1)$. 

For proving \eqref{q_abs} let $\xi_m$ be the dominant singularity of $\delta_m(z)$, which is the smallest positive singularity of $1 - 4 z^{a+d} \sum_{j=0}^{m-1} z^{b j}$. 
Observe that for all $j \geq 0$ we have $\xi_{j} > \xi_{j+1}$. 
Therefore the dominant singularity of $R_{m,q}(z)$ is $\xi_{m+q-1}$ and it is further away from
the origin than $\xi_{m+q}$ the dominant singularity of the first term in the right-hand side
of~\eqref{zq:eq:2}.
Hence the dominant contribution to the asymptotics of $L_{m,q,n} = [z^n] L_{m,q}(z)$ comes from the singularity at $\xi_{m+q}$ of type $\frac12$. 
So we easily get that (similar as in the proof of Proposition~\ref{i:prop:1}) 
\[
L_{m,q}(z) = 
R_{m,q}(\xi_{m+q}) - C \left(1 - \frac{z}{\xi_{m+q}}\right)^{\frac12} + \BigO{\left|1 - \frac{z}{\xi_{m+q}}\right|},
\] 
where
\[
C = \frac{\xi_{m+q}^{c q - d} \sqrt{\xi_{m+q}^{a+d} \sum_{j=0}^{m+q-1} (a+d+bj) \xi_{m+q}^{b
j}}}{\prod_{i=0}^{q-1} \delta_{m+i}(\xi_{m+q})}. \qedhere
\]
\end{proof}

\subsection{Lambda terms containing at most $q$ abstractions}

Let $L_{m,\leq q}(z)$ denote the generating function for lambda terms with at most $q$ abstractions.
If $q=0$, for all $m \geq 0$ we get once more the generating function for binary plane trees with sequence of successors of length less than $m$ and zero attached as leaves: $L_{m,\leq 0}(z) = L_{m, 0}(z)$.
Otherwise, $L_{m,\leq q}(z) = \sum_{l=0}^{q} L_{m,q}(z)$ and we can apply the results that we have obtained for a fixed number of abstractions.
The dominant singularity of $L_{m,\leq q}(z)$ comes from $L_{m,q}(z)$, whereas the terms $L_{m,l}(z)$ for $l < q$ give negligible contributions to the asymptotics.
So the terms with exactly $q$ abstractions outnumber those with at most $q-1$ abstractions and determine the asymptotic behaviour of the number of lambda terms with at most $q$ abstractions.

\section{Number of terms in normal form}
\label{normal} 

An important rule in lambda calculus is the so-called $\beta$-reduction, which reads follows:
$$(\lambda x.M)\: N \quad \implies \quad M[x\to N].$$ In words, the application of
the term $\lambda x.M$ to some term $M$ can be reduced to the term $M$ where all occurrences
of the bound variable $x$ have been replaced by a copy of $N$. $\beta$-reductions can we performed
on every subterm. A lambda term containing a subterm on which a $\beta$-reduction can be performed
is called a $\beta$-redex, otherwise it is called to be in normal form. In lambda calculus one is
interested in terms being in normal form. Thus we are keen to know how many of all lambda terms of
given size are also in normal form. 

From the discussion above it is clear that a lambda term is a $\beta$-redex if it contains a
subterm of the form $\cA\times (\cU\times \cL) \times \cL$. Let $\cB_m$ denote the class of all
$m$-open lambda terms which are not a $\beta$-redex and $\cB_\infty$ the class of all lambda terms
not being a $\beta$-redex. The exclusion of subterms of the above described form yields the
specification 
\begin{align*} 
\cB_m &= \seq_{\le m-1} (\cS) \times \cZ \;\cup\; \cU  \times \cB_{m+1} \;\cup\; \cA \times
(\cA\times \cB_m^2 \;\cup \; \seq_{\le m-1} (\cS) \times \cZ \times \cB_m), 
\end{align*} 
and from there we can derive directly the system of functional equations 
\begin{align*}      
B_m(z) &= z^a\sum_{j=1}^{m-1} z^{bj} + z^cB_{m+1}(z) + z^d \((z^d B_m(z)^2)\cdot B_m(z) + z^a
\sum_{j=1}^{m-1} z^{bj} B_m(z)\) \\
&= z^a\frac{1-z^{bm}}{1-z^b} + z^cB_{m+1}(z) + z^{a+d}\frac{1-z^{bm}}{1-z^b} B_m(z) + z^{2d}
B_m(z)^3 
\end{align*} 
for the associated generating functions. 

In an analogous way as for $L_m(z)$ it can be shown that all the functions $B_m(z)$ have the same
dominant singularity and that it equals the dominant singularity of $B_infty(z)$. 

For $m=\infty$ we have a similar functional equation. Putting all the terms on one side yields: 
\[
z^{2d} B_\infty(z)^3 + \(\frac{z^{a+d}}{1-z^b}+z^c-1\) B_\infty(z) + \frac{z^a}{1-z^b} = 0.
\]
This equation can be solved for $B_\infty(z)$ and we obtain three solutions from Cardano's
formula. The simplest is   
\begin{equation} \label{solB}
\frac{\sqrt[3]{\sqrt{X(z)^2 + \frac1{27}\(X(z)-4(1-z^c)\)^3}-X(z)}}{2z^d} - 
\frac{X(z)-4(1-z^c)}{6z^d \sqrt[3]{\sqrt{X(z)^2 + \frac1{27}\(X(z)-4(1-z^c)\)^3}-X(z)}}
\end{equation}  
where 
\[
X(z)=\frac{4z^{a+d}}{1-z^b}. 
\]
The other two solutions differ from this one by a third root of unity as multiplicative factor.
Thus it is sufficient to analyze \eqref{solB} to get an estimate for the fraction of lambda terms
in normal form. 

Obviously, singularities of the expression \eqref{solB} can only appear if one of the radicands
becomes zero. We know that $\rho$ must be a lower bound for the dominant singularity of
$B_\infty(z)$ and, as there are terms in normal form of any size, 1 is an upper bound. Moreover,
for $0\le z \le\rho$ we have $X(z)\le (1-z^c)^2$, where equality holds only
for $z=\rho$. This implies that for $0< z\le \rho$ we have 
\begin{align*} 
X(z)^2 + \frac1{27}\(X(z)-4(1-z^c)\)^3 & \le (1-z^c)^4 + \frac1{27}\((1-z^c)^2-4(1-z^c)\)^3 \\
&= (1-z^c)^4+ (1-z^c)^3 \frac{(-3-z^c)^3}{27} \\
&= (1-z^c)^3 \( 1-z^c - \(1+\frac{z^c}3\)^3\) < 0.
\end{align*} 
As the first inequality above is strict if $z<\rho$, this implies that $X(z)<0$ on the whole
interval $[0,\rho]$. Consequently, we also have 
\[
27 \sqrt{X(z)^2 + \frac1{27}\(X(z)-4(1-z^c)\)^3}-27X(z) < 0 
\]
for $0\le z\le \rho$, which means that the dominant of $B_\infty(z)$ must be strictly larger than
$\rho$. 

From \eqref{solB} we observe that $B_\infty(z)$ is singular if either 
\begin{equation} \label{firstequ} 
\sqrt{X(z)^2 + \frac1{27}\(X(z)-4(1-z^c)\)^3}-X(z) = 0
\end{equation} 
or 
\begin{equation} \label{secondequ}  
X(z)^2 + \frac1{27}\(X(z)-4(1-z^c)\)^3 = 0.
\end{equation} 
Note that we are searching for solutions of \eqref{firstequ} or \eqref{secondequ} which lie in the
interval $(\rho,1]$ and 
that $X(z)$ is a positive and strictly increasing function on the interval $[0,1)$. 
This implies that \eqref{firstequ} is equivalent to $X(z)=4(1-z^c)$. Therefore, the smallest
positive solution $\tilde\rho$ of \eqref{secondequ} is smaller than that of \eqref{firstequ}. 

The exponential term of the asymptotic number of lambda terms in normal form is determined by
$\tilde\rho$. In order to find out the subexponential factor, we 
consider the second term of the Taylor expansion of 
\begin{equation} \label{definef}
f(z):=X(z)^2 + \frac1{27}\(X(z)-4(1-z^c)\)^3
\end{equation} 
at $z=\tilde \rho$. The derivative is 
\begin{align} 
f'(z)&= 2X(z)X'(z)+ \frac19\(X(z)-4(1-z^c)\)^2 (X'(z)+4cz^{c-1}) \\ 
&=2X(z)X'(z)+X(z)^{4/3} (X'(z)+4cz^{c-1}) \label{ableitung}
\end{align} 
where we recalled that $X(z)>0$ as well as $X'(z)>0$ and thus the square
$\frac19\(X(z)-4(1-z^c)\)^2$, which is obviously positive, must be equal to $X(z)^{4/3}$. The last
expression \eqref{ableitung} consists of positive terms only, hence $f'(\tilde\rho)\neq 0$. Thus,
locally around $z=\tilde\rho$ the singularity in \eqref{solB} stems from the square-roots under
the third roots, the third roots themselves being regular and nonzero. This leads again to a
singularity of type $\frac12$, provided that there are no cancellations. 

To show this, observe that $\sqrt{f(z)}\sim \sqrt{f'(\tilde\rho)(z-\tilde\rho)}=:R$. Obviously, 
$R$ has a singularity of type $\frac12$, and we are supposed to show that this behaviour propagates
to the function given in \eqref{solB}. Setting 
\[ 
X:=X(\tilde\rho) \text{ and } Y:=\frac{4(1-\tilde\rho^c)-X(\tilde\rho)}{3} 
\] 
and keeping in mind that $X^2=Y^3$, the function \eqref{solB} is asymptotically equivalent to 
\begin{align*}
\frac{\sqrt[3]{R-X}}{2\tilde\rho^d}-\frac{Y}{2\tilde\rho^d \sqrt[3]{R-X}} &=
\frac{1}{2\tilde\rho^d}\(-\sqrt[3]{X} \sqrt[3]{1-\frac RX} -\frac{Y}{-\sqrt[3]{X} \sqrt[3]{1-\frac
RX}} \) \\ 
&=\frac{1}{2\tilde\rho^d}\(-\sqrt{Y} \sqrt[3]{1-\frac RX} +\frac{\sqrt{Y}}{\sqrt[3]{1-\frac
RX}}\) \\ 
&=\frac{\sqrt{Y}}{2\tilde\rho^d} \(-\(1-\frac RX\)^{1/3} + \(1-\frac RX\)^{-1/3} \) \\
&=\frac{\sqrt{Y}}{2\tilde\rho^d} \(\frac{2}{3X}\cdot R + O(R^2)\).  
\end{align*} 
Since $\frac{2}{3X}>0$ this shows that we have indeed a singularity of type $\frac12$. Summarizing
all our considerations leads to the following result.

\begin{theorem}
Let $\tilde\rho$ be the smallest positive root $f(z)$, the function defined in \eqref{definef}. 
Then there exists a positive constant $C$ (depending on $a, b, c, d$ and $m$) such that the number
of $m$-open lambda terms of size $n$ which are in normal form satisfies
\[
\left[z^n\right] B_{m}(z) \sim C n^{-\frac32} \tilde\rho^{-n}, \text{ as } \nti.
\]

Moreover, if $\rho$ is as in Theorem~\ref{thm:1}, then $\tilde\rho>\rho$. Therefore, the fraction
of lambda terms in normal form in the set of $m$-open lamda terms of size $n$ is exponentially
small, as $n$ tends to infinity, and the rate at which this fraction tends to zero is
$\Theta\( \(\frac{\rho}{\tilde\rho}\)^n\)$. 
\end{theorem}

\begin{remark}
For concretely given $a,b,c,d$ the singularities $\rho$ and $\tilde\rho$ can be approximated
numerically. For instance, for natural counting ($a=b=c=d=1$) we have $\rho\approx 0.295598$ and
$\tilde\rho\approx 0.318876$, which shows that the fraction of terms in normal form among terms of
size $n$ is $\Theta(\alpha^n)$ with $\alpha\approx 0.926999$. For binary lambda terms we get
$\rho\approx 0.509308$ and $\tilde\rho\approx 0.526219$. Thus the fraction of terms in normal form
is $\Theta(\beta^n)$ with $\beta\approx 0.967864$.
\end{remark}

\section{A random sampler for lambda terms}\label{Boltzmann}

In this section we will discuss the random generation of lamdba term, where we require that, if
conditioning on a specified size, the sampler draws according to the uniform distribution. It is
well-known that Boltzmann samplers have exactly this property and that they are well-suited to a
generating function approach. For simplicity, we will only treat the natural counting case
($a=b=c=d=1$) in this section. However, all our results also apply to the general case.

\subsection{Introduction to singular Boltzmann sampling}

Boltzmann samplers have been introduced by Duchon \emph{et al.}~\cite{DFLS04} in 2004. It is a universal
framework to create automatically a sampler for objects belonging to any specified combinatorial
class. All the constructors of the symbolic method have been "`interpreted"' in terms of samplers.
For instance, a Boltzmann sampler for the product of two classes $\mathcal{A} \times \mathcal{B}$
returns a couple obtained by Boltzmann sampling from each class ($\Gamma (A\times B)=(\Gamma
A,\Gamma B)$). Contrary to the recursive method, these samplers do not need to pre-compute
the coefficients of the series associated to the class, but directly deal with the generating
function which has to be evaluated. These new samplers do not output an object of a specified
size, but one considers the possible outputs of a given size, then the object is drawn uniformly
from all objects of same size. In other word, the conditional distribution, when conditioned on the
size of the output, is the uniform distribution on the set of objects of the desired size. This is
in general enough in most applications. 

In the case of singular Boltzmann samplers, the distribution of the output size is governed by
the Boltzmann distribution, \emph{i.e.}, 
$Prob(N=n)=\dfrac{a_n \rho^n}{A(\rho)}$ where $A(z)$ is the generating function of the
combinatorial class which is sampled and $\rho$ is the dominant singularity of $A(z)$.

\subsection{The sampler}

The idea here is to define a superclass $\cL_{N,0}$ of the class $\cL_0$ of closed lambda terms,
then sample lambda terms from this class and reject them if they are not closed. 

We start to define $\cL_{N,0}$, $\cL_{N,1}$ and so forth by using the specification
\eqref{m_open_def}, but for $\cL_{N,N}$ we drop the constraint on the leaves. The generating
functions associated with these classes satisfy therefore the following system of functional
equations: 

\begin{equation} \label{boltz}
\begin{cases}
    L_{N,0} &= z L_{N,1} + z L_{N,0}^2 \enspace , \\
    L_{N,1} &= z L_{N,2} + z L_{N,1}^2 + z \enspace , \\
    L_{N,2} &= z L_{N,3} + z L_{N,2}^2 + z + z^2 \enspace , \\
    \ldots &= \ldots \enspace , \\
    L_{N,N-1} &= z L_{N,N} + z L_{N,N-1}^2 + z \dfrac{1 - z^{N-1}}{1 - z} \enspace ,\\
    L_{N,N} &= z L_{N,N} + z L_{N,N}^2 + \dfrac{z}{1 - z} \enspace
\end{cases}
\end{equation}

By dropping the constraint on the leaves in $\cL_{N,N}$ we may have De Bruijn indices which are
too large for a term in $\cL_{N,N}$ to be $N$-open, and this propagates down to $\cL_{N,0}$. But
clearly, in order to violate the closedness condition when $N$ is large, we need a large De Bruijn
index and thus a term of large size. Hence the generating function associated with $\cL_{N,0}$
converges to $L_0(z)$, both coefficient-wise as well as uniformly on the interval $[0,\rho]$.

On the level of Boltzmann sampling we will sample a term from $\cL_{N,0}$ and reject it if it is
not closed. To get even more efficient, we perform some (controlled) extra rejections to obtain
closed lambda terms. 
In detail, let us denote by $\Gamma L_k$ the sampler which draws objects from $\cL_{N,k}$. The
samplers $\Gamma L_0,\dots,\Gamma L_N$ are automatically obtained from the specification which 
led to \eqref{boltz}. But the procedure actually has to call $\Gamma L_N$ before it can complete
the calls of the lower order samplers. If a De Bruijn index larger than
$N$ is drawn in $\Gamma L_N$, then we add a further rejection, since we will never get a closed
term when pushing the output through all lower order samplers $\Gamma L_k$, $k<N$.

Let us state some preliminary remarks before we analyze this sampler. Without this extra rejection,
the sampler lives in the traditional framework of singular Boltzmann sampling. Consequently, we
can apply the complexity theorems about this algorithm which state that such a sampler is linear
in the size of the output. If we only keep objects drawn in some window $](1-\varepsilon) n,
(1+\varepsilon) n[$, then the sampler stays linear \cite{BGR15,DFLS04}.

So, we essentially have to study the impact of the extra-rejection. A first observation is 
that in $\Gamma L_N$, the De Bruijn indices are drawn according to independent geometric
laws of parameter $\rho<1$. So, the average De Bruijn index drawn is $\dfrac{\rho}{1-\rho}$ 
(which is quite small). Consequently, the probability to draw an unbound variable, \emph{i.e.}, a
De Bruijn index which is too large (larger than $N$), is very low for large $N$. 

The next section is dedicated to a precise analysis of this fact.

\subsection{The complexity analysis}

We are going to give an upper bound for the cost of the extra rejection by assuming that 
the costs for each rejected object in proportional to its size. In fact, rejecting is cheaper, 
because we stop the building process as soon as the first unbound variable is observed. 
But it turns out that this upper bound is sufficient, since we will eventually reject an
arbitrarily small number of generated terms by choosing $N$ large enough.

\begin{figure}[h]
\begin{center}
\includegraphics[scale=0.4]{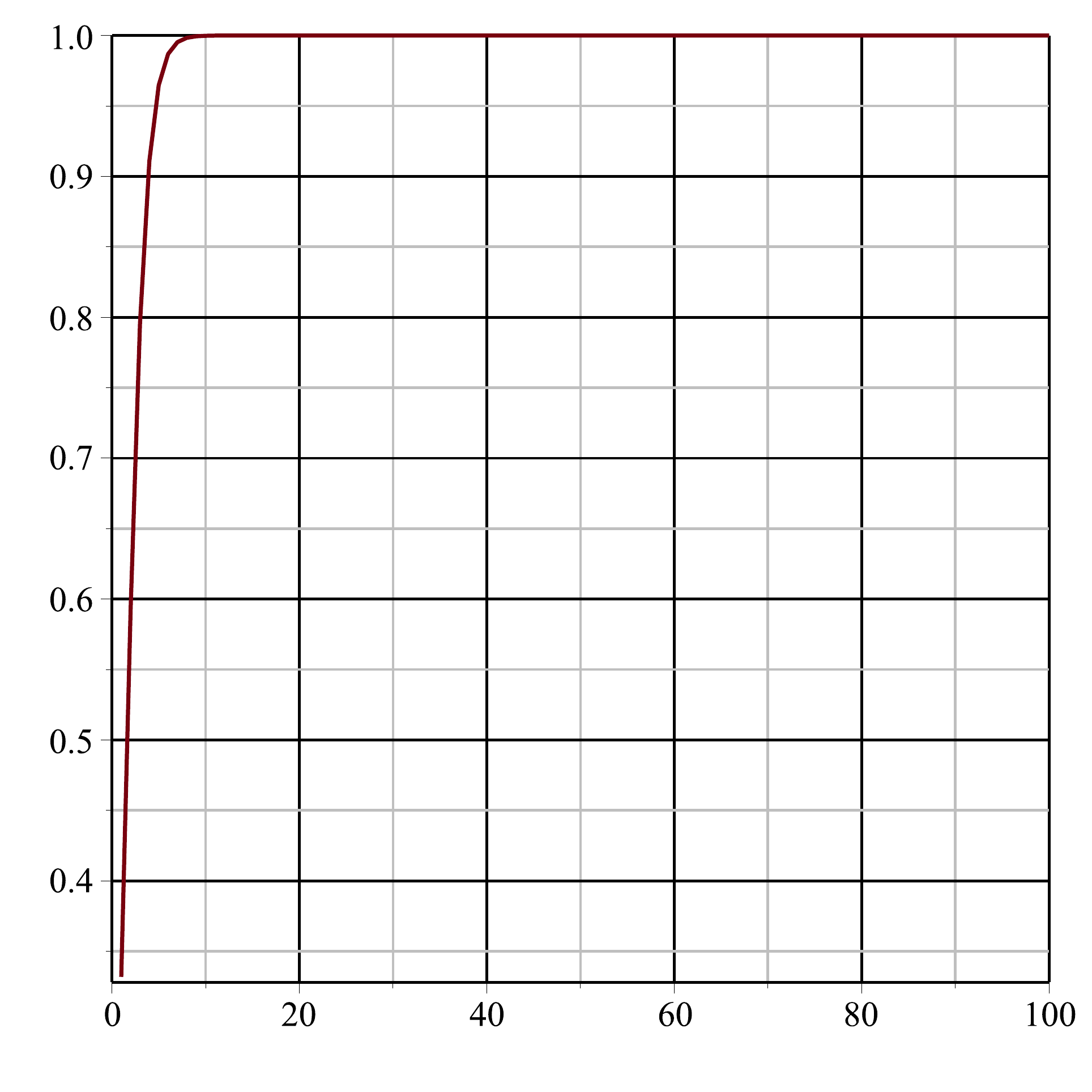}
   \caption{Proportion of closed terms in $\mathcal{L}_{N,0}$ for $N$ from 1 to 100}
   \end{center}
\end{figure}

For this purpose, let us determine the proportion of closed lambda terms of size $n$ in the class
$\mathcal{L}_{N,0}$. This proportion is just 
$\dfrac{[z^n]L_{0}(z)}{[z^n]L_{N,0}(z)}$. But we know that $L_{N,0}(z)$ has the same radius of
convergence as $L_{0}(z)$ and that the sequence converges uniformly to $L_{0}(z)$. Even a more
precise result holds, which can be seen by inspecting the system \eqref{boltz}. 

\begin{proposition}\label{bnzero}
For $m=0,1,\dots,N$ there are positive constants $a_{N,m}$ and $b_{N,m}$ such that $L_{N,m}(z)\sim
a_{N,m}-b_{N,m} \sqrt{1-\frac z\rho}$, as $z\to\rho$ in such a way that $\arg(z-\rho)\neq 0$. 
\end{proposition}

\begin{proof}
The last equation of \eqref{boltz} is the quadratic equation for $L_{N,N}(z)=L_\infty(z)$. Solving
it yields the explicit solution for $L_{N,N}(z)$ and gives in particular
\[
L_{N,N}=\frac{1-z-\sqrt{\frac{1-3z-z^2-z^3}{1-z}}}{2z}\sim a_{N,N}-b_{N,N} \sqrt{1-\frac z\rho}
\]
with $\rho$ being the smallest positive root of $1-3z-z^2-z^3$ and 
\[
a_{N,N}=\frac{1-2\rho}{2\rho} \qquad b_{N,N}=\frac{\sqrt{3-6\rho-\rho^2}}{2\rho\sqrt{1-\rho}}.
\]
Using the other equations of \eqref{boltz}, solving successively from bottom up, we obtain nothing
but solutions of quadratic equations. Inserting the singular expansions one after another we
obtain a system of two recurrence relations for $a_{N,m}$ and $b_{N,m}$: 
\begin{align*} 
a_{N,m}&= \frac{1-\sqrt{1-(1-\rho^2)(1-\rho^m)-4\rho^2 a_{N,m+1}}}{2\rho}, \\
b_{N,m}&=\frac{\rho}{\sqrt{1-(1-\rho^2)(1-\rho^m)-4\rho^2 a_{N,m+1}}} \cdot b_{N,m+1}.
\end{align*}
From this system it is obvious that $b_{N,m}>0$ and hence all function $L_{N,m}(z)$ have indeed a
singularity of type $\frac12$. The positivity of $a_{N,m}$ is trivial, because
$a_{N,m}=L_{N,m}(\rho)$, $\rho>0$ and $L_{N,m}(z)$ has positive coefficients. 
\end{proof}

Now, let us return to our sampling problem and note that
$\dfrac{[z^n]L_{0}(z)}{[z^n]L_{N,0}(z)}=\dfrac{b_0}{b_{N,0}}$. 
Proposition~\ref{bnzero} gives us a procedure to
compute $[z^n]L_{N,0}(z)$ asymptotically. The next proposition will tell us that we do not have to
go very far, \emph{i.e.}, an $N$ of moderate size is sufficient.

\begin{proposition}
The fraction $\dfrac{b_0}{b_{N,0}}$ tends to 1 exponentially fast, as $N$ tends to
infinity. 
\end{proposition}

\begin{proof}
Consider the difference $L_{m,N}(z)-L_m(z)$. We have 
\[
L_{m,N}(z)-L_m(z) = z(L_{m+1,N}(z)-L_{m+1}(z)) + z(L_{m,N}(z)^2-L_m(z)^2, 
\]
for all $m=0,1\dots,N-1$, which implies 
\[
L_{m,N}(z)-L_m(z) = \frac{z}{1-z(L_{m,N}(z)+L_\infty(z))} (L_{m+1,N}(z)-L_{m+1}(z)).
\]
Iterating gives 
\[
L_{0,N}(z)-L_0(z) = (L_\infty(z)-L_N(z)) 
\prod_{\ell=0}^{N-1} \frac{z}{1-z(L_{\ell,N}(z)+L_\ell(z))}. 
\]
The function on the left-hand side has only nonnegative coefficients and therefore, in order to
estimate the function it is enough to consider only positive $z$. For positive $z$ we have
$L_{m,N}(z)< L_\infty(z)$ and $L_m(z)< L_N(z)$. Thus 
\begin{align} 
|L_{0,N}(z)-L_0(z)|&< |L_\infty(z)-L_N(z)| \prod_{\ell=0}^{N-1} \left|
\frac{z}{1-z(L_{\infty}(z)+L_N(z))} \right|  \nonumber \\
&= |L_\infty(z)-L_N(z)| \cdot \left| 
\frac{z}{1-2zL_{\infty}(z)+z(L_infty(z)-L_N(z))} \right|^N. \label{auxil_ineq}
\end{align} 

By an analogous reasoning as in the proof of Lemma~\ref{speedlemma} we can derive that
$|L_\infty(z)-L_N(z)|=O(N^{-1/2})$. This can be refined to
$|L_\infty(\rho)-L_N(\rho)|=\Theta(N^{-1/2})$. To see this, note that as $b_N<b_\infty$, the square-root
singularities of $L_\infty(z)$ and $L_N(z)$ cannot cancel when building simply the difference.
Thus $L_\infty(z)-L_N(z)$ has a singularity of type $\frac12$ as well and thus its coefficients
satisfy an asymptotic of the form $\rho^{-n}n^{-3/2}$, and they are zero only for $n<N$. So, the
modulus of the difference must be of order $N^{-1/2}$ when evaluated at $z=\rho$. 

With this in conjunction with $1-2\rho L_\infty(\rho)=\rho$ the inequality \eqref{auxil_ineq}
becomes   
\[
|L_{0,N}(z)-L_0(z)|\le |L_{0,N}(\rho)-L_0(\rho)|< 
\frac{c_1}{\sqrt N} \(\frac{C}{1+\frac{c_2}{\sqrt N}}\)^N 
\]
which means that the difference $|L_{0,N}(z)-L_0(z)|$ is exponentially small. Likewise, 
$(b_{N,0}-b_0)/b_{N,0}$ is exponentially small, because the coefficients $b_0$ and $b_{N,0}$ can
be expressed as integrals of $L_{0,N}(z)$ and $L_0(z)$ by means of Cauchy's integration formula. 
\end{proof}

Obviously, our results generalize to the general case with arbitrary $a,b,c,d$ as well. Exploiting
these results for the present case ($a=b=c=d=1$), shows the efficiency of the sampler: Already for
$N=20$, the proportion of closed terms is 0.999999998 such that the sampler will almost never
reject a drawn object. 

\subsection{Experiments}

In this section, we deal with a Boltzmann sampler using $N=20$. It is quite interesting to
analyze the behavior of this sampler. The first graphic shows the dynamic of the choice inside the
system of equations. The first observation is that this dynamic tends quickly to a stable law: the
probability to draw a unary node tends to 0.2955977425, the probability to draw a binary node as
the probability to draw a leaf tends to 0.3522011287. In particular, this fact shows that the
number of leaves in a lambda term of size $n$ is asymptotically of order $\mu n$ with $\mu\approx 0.3522011287$.

\begin{figure}[h]
\begin{center}
\includegraphics[scale=0.4]{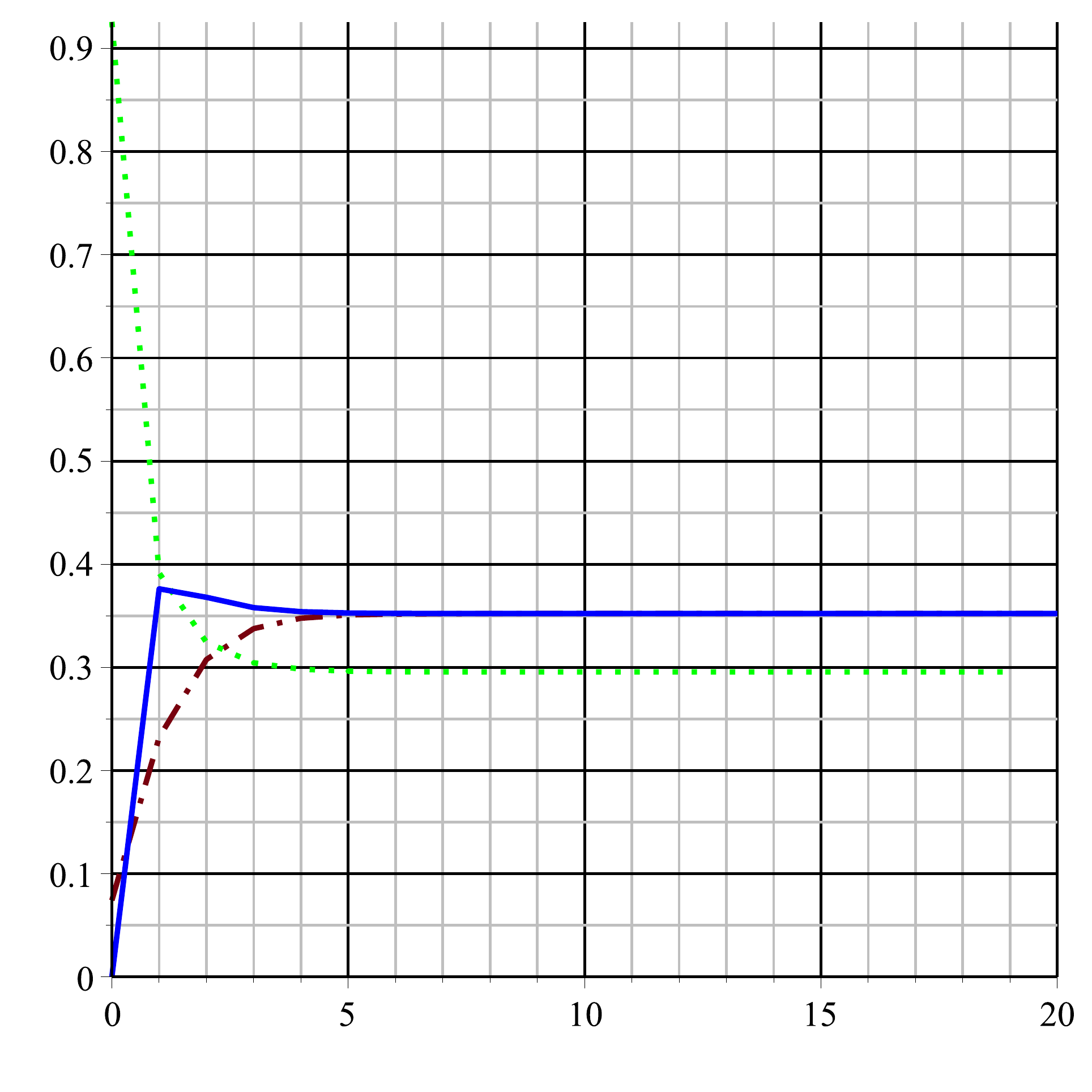}
   \caption{Example for $N=20$ of the evolution of the probability in $\Gamma L_i$ for  $i$ from 0
to 20. The green dot line represents the probability to draw a unary node (abstraction), the red
dashdot line represents the probability to draw a binary node (application), the solid blue line
represents the probability to draw a leaf (variable).}
   \end{center}
\end{figure}

Invoking the uniform convergence, we can reach more precise results on parameters. In particular,
consider $X_N$ the random variable of the number of leaves in a uniform random term from
$\mathcal{L}_{N,0}$, conditioned on the size. 

\begin{figure}[h]
\begin{center}
\includegraphics[scale=0.25]{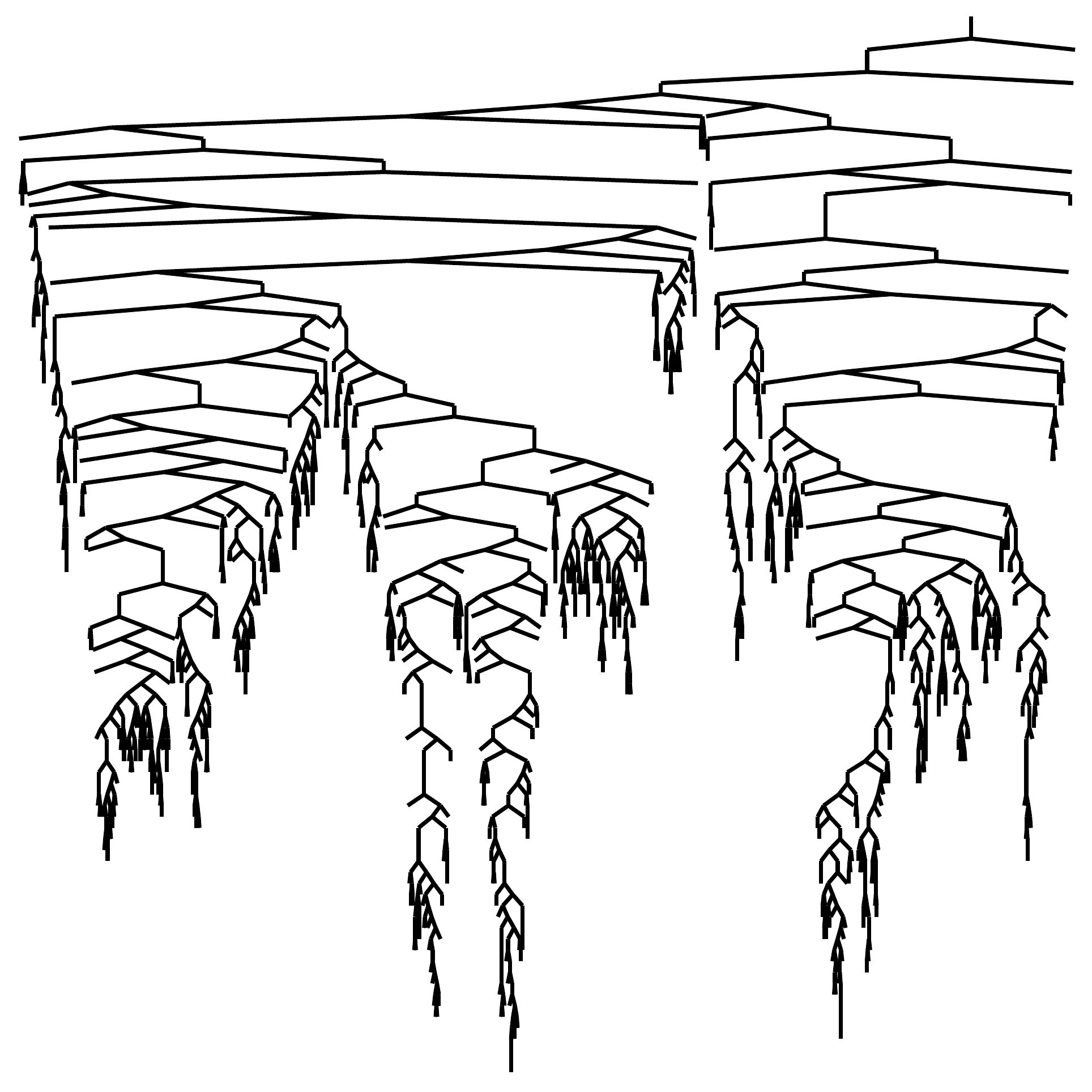}
\includegraphics[scale=0.25]{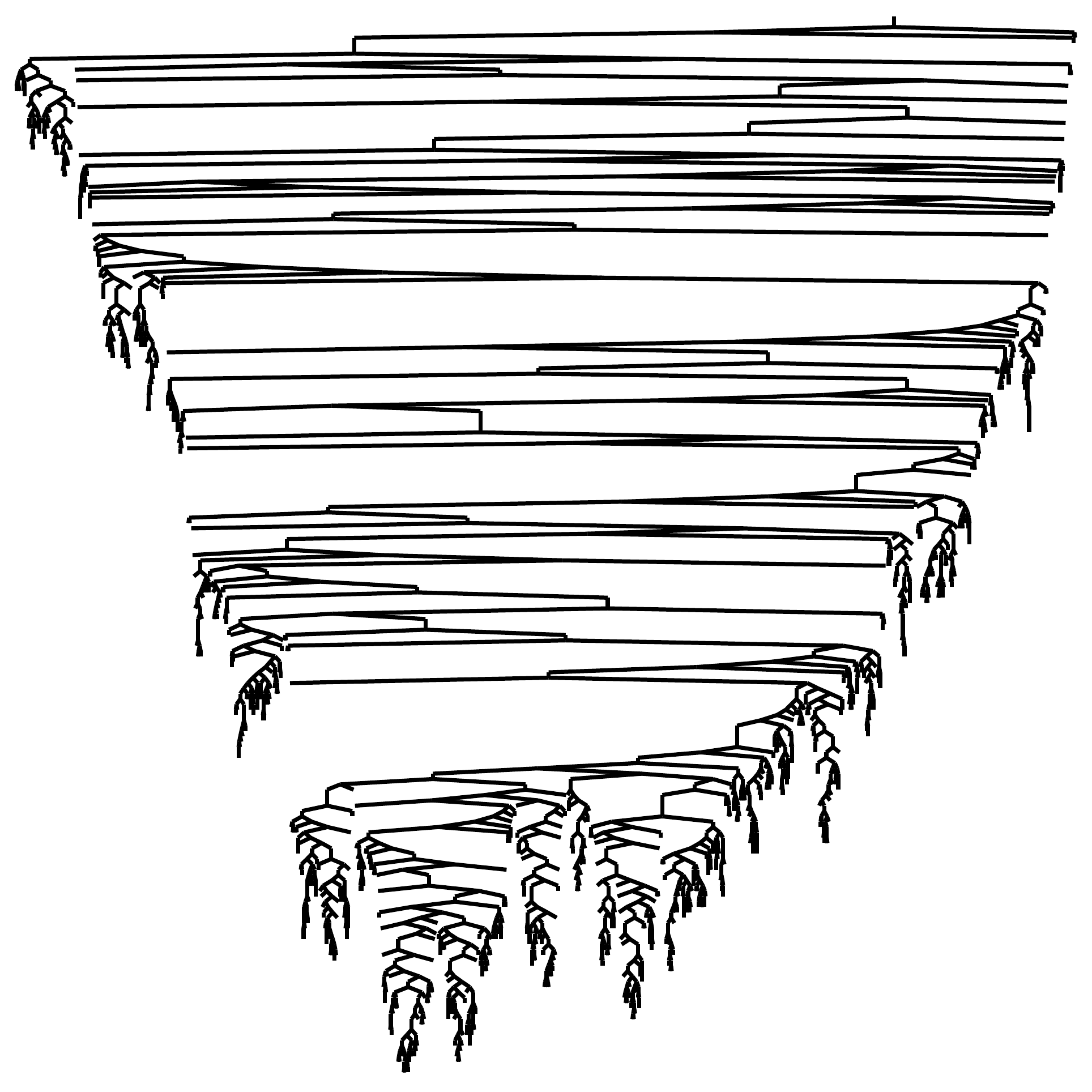}
\includegraphics[scale=0.25]{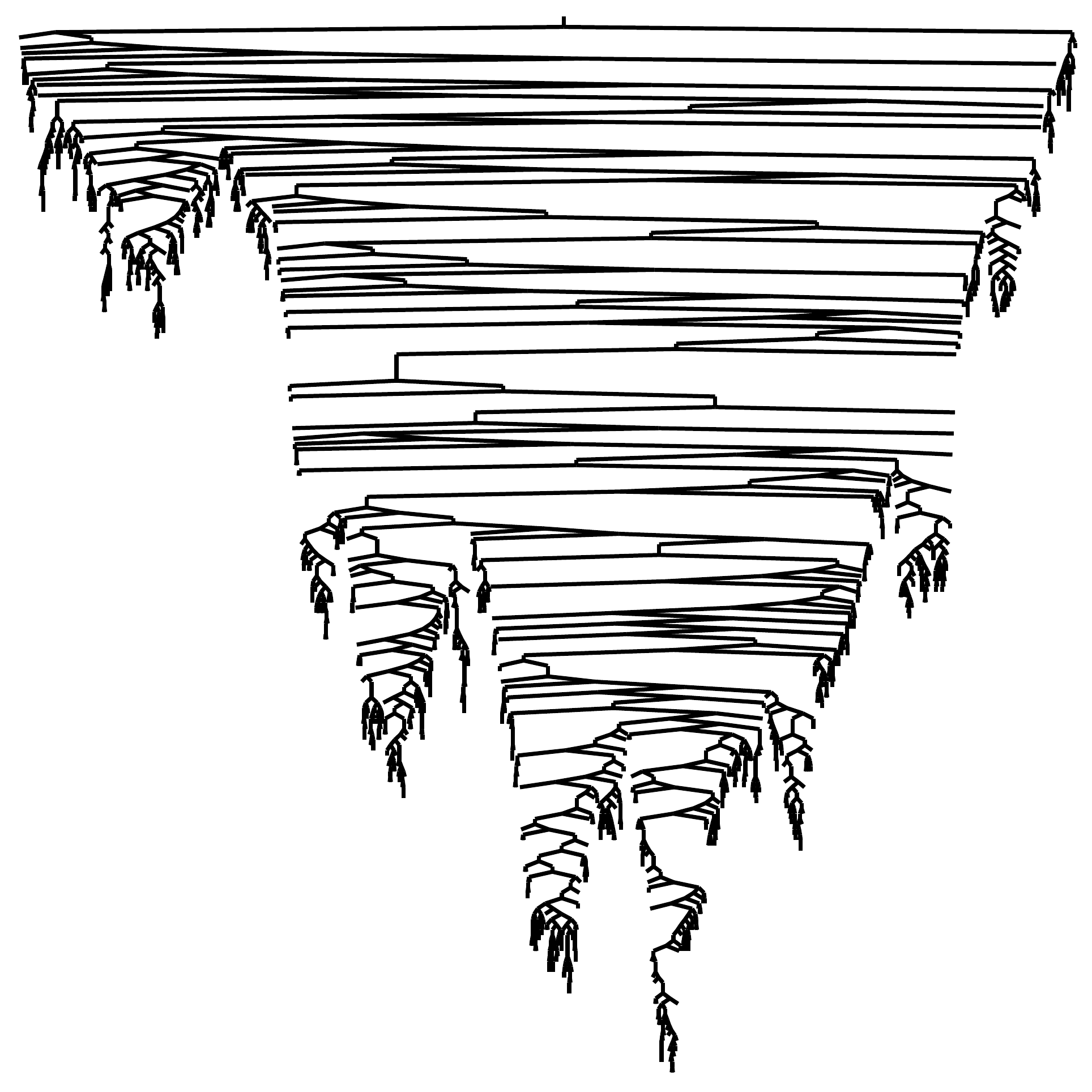}
   \caption{Three uniform random lambda terms (according to unary representation) of size
respectively 2098, 2541, 2761.} \label{largeterms}
   \end{center}
\end{figure}

First, notice that $X_N$ has the same distribution as the number of leaves
in $\mathcal{L}_{N,N}$. This can be seen by setting up equations for the bivariate generating
functions $L_{N,k}(z,u)$ where the second variable is associated to the leaves. Then, one observes
that the radius of convergence $\rho_N(u)$ of all functions is the same and depends only on $N$.
As $L_{N,0}$ converges uniformly to $L_0(z)$ the same applies to $\rho_N(u)$. Eventually, this
implies that all moments of $X_N$ are asymptotically equal to the corresponding moments of $X$, 
the number of leaves in a uniform random closed term, conditioned on the size. 
It follows that $X$ is asymptotically Gaussian as it
is the case for the number of leaves in $\mathcal{L}_{N,N}$.

\begin{theorem}
Let $X_n$ the number of variables in a lambda term of size $n$. Then $X_n$ is asymptotically 
Gaussian with $\mathbb{E}(X_n)\sim\alpha n$ and $\mathbb{V}ar(X)\sim \alpha n$ where 
$\alpha=\dfrac{1-\rho}{2}\approx 0.3522011287$. 
\end{theorem}

For sampling lambda terms several approaches were use. In \cite{Wa05} an algorithm using
memoization techniques was developed. The Haskell sampler in \cite{GL13} which used recurrences
for different classes of lambda terms performed better and could generate terms of size 1500
within a few hours. 

Our sampler has the same complexity as the classical Boltzmann sampler for trees. In particular, it
is linear in approximate size. Grygiel and Lescanne \cite{GL15} already used a Boltzmann sampler
using a rejection procedure based on the results presented in \cite{GG16}. They were able to 
generate terms up to size 100 000. Our procedure refines the rejection procedure and exploits the
fast convergence of the $\cL_{N,0}$ to $\cL_0$. Using a standard laptop with a CPU i7-5600U
cadenced at 2.6 GHz, it is possible to draw a lambda term of size in the range [1 000 000,2 000
000]. The fairly large terms depicted in Figure~\ref{largeterms} can be generated within a
second.

\bibliographystyle{plain}
\bibliography{bibliography}

\end{document}